\documentclass[11pt, a4paper]{article}

\usepackage{lipsum}
\usepackage{geometry}
\usepackage{authblk}
\usepackage{bm}
\usepackage{amssymb, amsmath}
\usepackage{color}
\usepackage[amsmath,thmmarks]{ntheorem}
\numberwithin{equation}{section}
\newcommand{\fnum}{\mathbb{F}}

\newtheorem{theorem}{Theorem}[section]

\newtheorem{lemma}{Lemma}[section]
\newtheorem{corollary}{Corollary}[section]

\theorembodyfont{\rmfamily}
\newtheorem{remarks}{Remark}[section]

\makeatletter
\newtheoremstyle{MyNonumberplain}%
  {\item[\theorem@headerfont\hskip\labelsep ##1\theorem@separator]}%
  {\item[\theorem@headerfont\hskip\labelsep ##3\theorem@separator]}
\makeatother
\theoremstyle{MyNonumberplain}
\theorembodyfont{\upshape}
\theoremsymbol{$\square$}
\newtheorem{proof}{Proof}

\begin{document}

\title{Simultaneous Diophantine Approximation in Function Fields}

\date{}

\author{Zhiyong Zheng\footnote{This work was partially supported by the ``973'' project 2013CB834205 of P.R. China.}}
\affil{School of Mathematics and Systems Science, \authorcr
Beihang University, Beijing, P.R. China \authorcr
zhengzhiyong@buaa.edu.cn \vspace{4mm}}

\maketitle
\begin{abstract}
There are abundant results on Diophantine approximation over fields of positive characteristic (see the survey papers \cite{ref:9,ref:n4}), but there is very little information about simultaneous approximation. In this paper, we develop a technique of ``geometry of numbers'' in positive characteristic, so that we may generalize some of the classical results on simultaneous approximation to the case of function fields. More precisely, we approximate a finite set of Laurent series by rational functions with a common denominator. In particular, the lower bound results we obtain may be regarded as a high dimensional version of the Liouville--Mahler Theorem on algebraic functions of degree $n$. As an application, we investigate binary quadratic forms, and determine the exact approximation constant of a quadratic algebraic function. Finally, we give two examples using continued fractions.
\end{abstract}

\noindent{\small {\bf 2010 Mathematics Subject Classification: }11J61, 11J70, 11J83}

\noindent{\small {\bf Key words: }Simultaneous Approximation, Laurent Series Field, Haar Measure}

\section{Introduction and Results Statement}
\label{sec:intro}

Let $\fnum_q$ be a finite field with $q$ elements of characteristic $p$, $K = \fnum_q[T]$ be the polynomial ring, $k = \fnum_q(T)$ be the rational function field, and $k_{\infty} = \fnum_q((\frac{1}{T}))$ be the formal Laurent series field. Let $v$ be the normalized exponent valuation with $v(\frac{1}{T}) = 1$, so that $v(\alpha)$ takes on integer values and $v(0) = \infty$. If $\alpha$ is an element in $k_{\infty}$, then $\alpha$ can be uniquely expressed as a Laurent series as follows
\begin{equation}
  \label{eq:1-1}
  \alpha = \sum_{i=n}^{+\infty} a_i\left(\frac{1}{T}\right)^i, \quad n\in \mathbb{Z}, a_i\in \fnum_q, \mbox{ and }a_n\neq 0,
\end{equation}
where $v(\alpha) = n$. We define the square bracket function $[\alpha]$ by
\begin{equation}
  \label{eq:1-2}
  [\alpha] = \sum_{i=n}^0a_i\left(\frac{1}{T}\right)^i,\mbox{ if }n\leq 0, \mbox{ and }[\alpha]=0, \mbox{ if }n>0,
\end{equation}
which is called the ``integral part'' of $\alpha$ as usual. It is easily seen that $[\alpha]\in K$, $[\alpha+\beta] = [\alpha]+[\beta]$, $[a\alpha] = a[\alpha]$ for all $a\in \fnum_q^*$, and $v(\alpha - [\alpha])\geq 1$. In fact, there is a unique polynomial $A = [\alpha]$, such that $v(\alpha - A)\geq 1$. We also have
\begin{equation}
  \label{eq:1-3}
  v(\alpha) = v([\alpha]), \mbox{ if } v(\alpha) \leq 0.
\end{equation}
The absolute value functions $|\alpha|$ and $||\alpha||$ are given by
\begin{equation}
  \label{eq:1-4}
  |\alpha| = q^{-v(\alpha)}, \mbox{ and }||\alpha|| = |\alpha-[\alpha]|, \alpha\in k_{\infty}.
\end{equation}
It is worth keeping in mind that $|0|=0$, $|a|=1$ for all $a\in \fnum_q^*$, $|\alpha\beta| = |\alpha|\cdot |\beta|$ and in particular, if $|\alpha| < |\beta|$, then $|\alpha|\leq \frac{1}{q}|\beta|$. For the double absolute function, we have $||\alpha+A|| = ||\alpha||$ for all $A\in K$, $||a\alpha|| = ||\alpha||$ for all $a\in \fnum_q^*$, $0\leq ||\alpha|| \leq \frac{1}{q}$, for all $\alpha\in k_{\infty}$, and $||\alpha||=0$ if and only if $\alpha\in K$. In particular we have
\begin{equation}
  \label{eq:1-5}
  ||\alpha+\beta|| \leq \max\{||\alpha||, ||\beta||\}, \mbox{ and } ||\alpha|| = \inf\limits_{A\in K}|\alpha-A|.
\end{equation}
Thus, $||\alpha||$ is the smallest distance from $\alpha$ to any element of $K$, and $[\alpha]$ is the nearest polynomial to $\alpha$.

The story of Diophantine approximation in $k_{\infty}$ goes back to E. Artin, who first introduced continued fractions over $k_{\infty}$ (see \cite{ref:2}). Continued fractions have been the primary tool of investigation in the positive characteristic case up to the current day. In 1949, Mahler \cite{ref:11} proved the following analogue of Liouville's theorem: For any element $\alpha \in k_{\infty}$, algebraic  of degree $n>1$ over $k$, there is a constant $c=c(\alpha)>0$, such that
\begin{equation}
  \label{eq:1-6}
  ||Q\alpha|| \geq c\cdot |Q|^{1-n}
\end{equation}
for all nonzero $Q\in K$. Moreover, as Mahler noted in \cite{ref:11}, the exponent $1-n$ in \eqref{eq:1-6} is best possible. Indeed, for his famous  example $\alpha := T^{-1} + T^{-p} + \cdots + T^{-p^h} + \cdots $ in $k_{\infty}$, we have that $\alpha$ is degree $p$ over $k$, and there are infinitely many polynomials $Q_n$, such that $||Q_n\alpha|| = |Q_n|^{1-p}$. In particular, this means that Thue's theorem does not hold in $k_{\infty}$, needless to say Roth's remarkable theorem. Later, a few mathematicians try to establish an analogue of Thue's theorem under certain conditions. For details, we refer the reader to \cite{ref:10,ref:12}. For a general background to material on Diophantine approximation in characteristic zero and in positive characteristic, the reader should consult \cite{ref:n1,ref:n2,ref:6,ref:n3}.

Certainly, there are fruitful results on Diophantine approximation in positive characteristic (see the survey papers \cite{ref:9,ref:n4}), but there is very little information about simultaneous approximation (see \cite{ref:nn2,ref:nn3,ref:n7,ref:n5}). The first significant progress on high dimensional approximation is due to K. Mahler \cite{ref:n7}, who established an analogue to Minkowski's second theorem on the geometry of numbers in positive characteristic (see also Chapter~9 of \cite{ref:n2}). For any convex body defined by a distance function with non-Archimedean valuation, Mahler showed \cite{ref:n7} that the product of the successive minima of the convex body is equal to a certain positive constant, and the constant was considered as the volume of the given convex body by Mahler. In \cite{ref:n1}, E. Bombieri and W. Gubler proved Minkowski's second theorem over the adeles (see Theorem~C.2.11 of \cite{ref:n1}), which may be regarded as an adelic version of the above assertion of Mahler. However, the volume defined by Mahler (see (20) of \cite{ref:n7}) is lacking an explicit bound for various types of distance functions, and thus there still exists a gap between Mahler's theorem and simultaneous Diophantine approximation in general. The key point of connection is Lemma~\ref{lem:2-4} below, known as Minkowski's linear forms theorem (see Theorem~III of Appendix B of \cite{ref:5}). In \cite{ref:n5}, deMathan only dealt with a pair of algebraic functions approximated by rational functions, in other words, he only obtained a special two dimensional result in the homogeneous case.

In this paper, we investigate simultaneous approximation in $k_{\infty}$ by using the Haar measure.  Our main results are the following.

\begin{theorem}\label{thm:1-1}
  Suppose that $\theta_1, \theta_2, \ldots, \theta_n$ are $n$ elements in $k_{\infty}$, and $H\in k_{\infty}$ with $|H|\geq 1$. Then there is a polynomial $x\in K$, such that
  \begin{equation}
    \label{eq:1-7}
    \max\limits_{1\leq i \leq n}||x\theta_i|| \leq (q|H|)^{-\frac{1}{n}}, \mbox{ and }1\leq |x| \leq |H|.
  \end{equation}
\end{theorem}

The above theorem is a special case of Theorem~\ref{thm:2-2} below, in fact we have proved there a stronger, namely a multiplicative statement. As a direct consequence (see Corollary~\ref{cor:2-2} below), if $\theta_1, \theta_2, \ldots, \theta_n$ are $n$ elements in $k_{\infty}$, then there are infinitely many polynomials $x\in K$, such that
\begin{equation}
  \label{eq:1-8}
  |x|^{\frac{1}{n}}\max\limits_{1\leq i \leq n}||x\theta_i||\leq q^{-\frac{1}{n}}.
\end{equation}

The exponent $-\frac{1}{n}$ in \eqref{eq:1-7} is best possible, since we also have the following lower bound estimates.

\begin{theorem}\label{thm:1-2}
  Suppose that $\theta_1, \theta_2, \ldots, \theta_n$ are $n$ algebraic elements in $k_{\infty}$ such that $k(\theta_1, \dots, \theta_n)$ is of degree $n+1$ over $k$, and $\{1, \theta_1, \ldots, \theta_n\}$ are linearly independent over $k$. Then there is a constant $\gamma > 0$ (depending only on $\theta_1, \theta_2, \ldots,\theta_n$) such that
  \begin{equation}
    \label{eq:1-9}
    \max\limits_{1\leq i \leq n}||x\theta_i|| \geq \gamma|x|^{-\frac{1}{n}}
  \end{equation}
for all nonzero $x\in K$. Equivalently, there is a constant $\gamma_1 >0$, such that
\begin{equation}
  \label{eq:1-10}
  ||x_1\theta_1 + x_2\theta_2 + \cdots + x_n\theta_n|| \geq \gamma_1(\max\limits_{1\leq i \leq n}|x_i|)^{-n}
\end{equation}
for all nonzero $n$-tuples $x=(x_1, x_2, \ldots, x_n)\in K^n$.
\end{theorem}

We may regard Theorem~\ref{thm:1-2} as a high dimensional version of the above theorem of Mahler; see \eqref{eq:1-6}. To explain this point of view, let $\alpha$ be a fixed algebraic element in $k_{\infty}$ of degree $n+1$, we put $\theta_i = \alpha^i~(1\leq i \leq n)$ in the above theorem, then by Theorem~\ref{thm:1-2}, it follows that there is a constant $\gamma > 0$ (depending only on $\alpha$) such that
\begin{equation*}
  \max_{1\leq i \leq n} ||x\alpha^i || \geq \gamma |x|^{-\frac{1}{n}}
\end{equation*}
for all nonzero $x\in K$. Clearly we get a more generalized result than the original result of Mahler (see \eqref{eq:1-6}). For more discussions on this special case, the reader should be referred to \cite[Section~2]{ref:nn1}.

Next, as an application, we focus on  quadratic algebraic functions, and seek to determine the approximation constant. Let $p>2$, and
$$f(x, y) = \alpha x^2 + \beta xy + \gamma y ^2, \quad \alpha, \beta, \gamma \in k_{\infty}$$
be a binary quadratic form of discriminant $\delta = \beta^2 - 4\alpha\gamma$, not a perfect square in $k$. Any quadratic algebraic element is a root of $f(x, 1) = 0$ for some such $(\alpha, \beta, \gamma)\in K^3$. We define
\begin{equation}
  \label{eq:1-11}
  \sigma := \sigma(f) = \inf\{|f(x, y)| :(x, y)\in K^2, \mbox{ and }(x, y)\neq 0\}
\end{equation}
and ($\theta \in k_{\infty}$)
\begin{equation}
  \label{eq:1-12}
  \tau(\theta) := \lim_{|Q|\rightarrow \infty}\inf\{|x|~||x\theta||: x\in K, \mbox{ and } |x|\geq |Q|\}.
\end{equation}
 $\tau(\theta)$ is called the approximation constant of $\theta$, as usual. It is easy to verify that $0 \leq \tau(\theta) \leq \frac{1}{q}$ (see Corollary~\ref{cor:2-1} below). We have
\begin{theorem}\label{thm:1-3}
  If $\alpha$ is a quadratic algebraic element over $k$, then
  \begin{equation}
    \label{eq:1-13}
    \tau(\alpha) = |\delta|^{-\frac{1}{2}}\sigma.
  \end{equation}
Furthermore, there are infinitely many polynomials $x\in K$, such that
\begin{equation}
  \label{eq:1-14}
  |x|~||x\alpha|| = \tau(\alpha).
\end{equation}
\end{theorem}

The approximation constant $\tau(\theta)$ given by \eqref{eq:1-12} is a standard notion in real number fields (see page~11 of \cite{ref:5}). According to \cite{ref:9,ref:n5,ref:n6}, we may generalize this definition as follows: Let $\theta_1, \ldots, \theta_n$ be $n$ elements in $k_{\infty}$, and $\lambda$ be a real number. We define
\begin{equation}
  \label{eq:1-15n}
  B(\theta_1, \theta_2, \ldots, \theta_n, \lambda) : = \lim_{|Q|\rightarrow \infty} \inf\{ |x|^{\lambda} \prod_{j=1}^n || x\theta_j||: x\in K \mbox{ and } |x| \geq |Q|\},
\end{equation}
and define an approximation exponent for $(\theta_1, \dots, \theta_n)$ by
\begin{equation}
  \label{eq:1-16n}
  \tau^* (\theta_1, \theta_2, \ldots, \theta_n) := \sup\{\lambda:  B(\theta_1, \theta_2, \ldots, \theta_n, \lambda)< \infty \}.
\end{equation}
It is easily seen that if any of the $\theta_i$ are in $k$, then $ B(\theta_1, \theta_2, \ldots, \theta_n, \lambda) = 0$ for all real numbers $\lambda$, and thus $\tau^*(\theta_1, \theta_2, \ldots, \theta_n) = +\infty$. If all $\theta_i \not \in k~(1\leq i \leq n)$, then, by Theorem~\ref{thm:1-1} we have
\begin{equation}
  \label{eq:1-17n}
   B(\theta_1, \theta_2, \ldots, \theta_n, 1) \leq \frac{1}{q}.
\end{equation}
It follows that $\tau^*(\theta_1, \theta_2, \ldots, \theta_n) \geq 1$ for any choice of $\theta_i$ in $k_{\infty}$, $1\leq i \leq n$. This assertion is  fundamental and is a new result for $n > 1$ in the case of finite characteristic. In \cite{ref:n5}, deMathan provided some extreme examples with $\tau^*(\theta_1, \theta_2) > 1$ for some $\theta_1$ and $\theta_2$ in $k_{\infty}$. In particular, for $p=2$,  he showed that $\tau^*(\theta_1, \theta_2) > 1$ for all quadratic algebraic elements over $k$ in $k_{\infty}$ (see Theorem~2 of \cite{ref:n5}). It was conjectured by Littlewood that $B(\alpha, \beta, 1) = 0$ for all $\alpha$ and $\beta$ in $k_{\infty}$ (see \cite{ref:n5}).

In the last section of this paper, we give two simple examples of quadratic algebraic functions for which the approximation constant $\tau(\alpha)$ is exactly $q^{-d}~(d=1, 2, \ldots)$, and the degrees of the partial quotients of the continued fraction are bounded by a constant. If $d=1$, we set
\begin{equation}
  \label{eq:1-15}
  \alpha(a, b) = \frac{1}{2}(\sqrt{(aT+b)^2+4} - (aT+b)), \quad a\in \fnum_q^*, b\in \fnum_q.
\end{equation}
   This is the  analogue of the real number $\frac{1}{2}(\sqrt{5}-1)$ (see \cite{ref:n9}), because $\tau(\alpha(a, b)) = \frac{1}{q}$. In the classical case, a real number $\alpha$ has approximation constant $\frac{1}{\sqrt{5}}$, if and only if $\alpha$ is equivalent to $\frac{1}{2}(\sqrt{5}-1)$. But for function fields, we do not have this assertion.

Throughout this paper, we denote by $k_{\infty}^n$ the $n$ dimensional linear space over $k_{\infty}$. If $x=(x_1, x_2, \ldots, x_n)\in K^n$, we call $x$ is an integral point of $k_{\infty}^n$. In addition, if $x\neq 0$, it is said to be a nonzero integral point. We denote by $[\alpha]$ the integral part of $\alpha$. If $\alpha$ is a real number, then $[\alpha]$ is the integral part of $\alpha$ as usual.

\section{Homogeneous Approximation}
\label{sec:2}

We start this section with some basic notions from harmonic analysis on $k_{\infty}$. For more details the reader should consult \cite[Chapter~2]{ref:7} and \cite{ref:n8}. Let $N$ denote the valuation ideal of $k_{\infty}$,
$$N = \{\alpha \in k_{\infty} : v(\alpha) \geq 1\}. $$
Since $k_{\infty}$ is a complete metric space with respect to the valuation $v$, and is a local compact topological additive group, there exists a unique, up to a positive multiplicative constant, countably additive Haar measure. Let $\mu$ denote the Haar measure normalized to $1$ on $N$, so that
\begin{equation}
  \label{eq:2.1}
  \mu(\{\alpha\in k_{\infty}: v(\alpha) \geq m\}) = q^{1-m}, \quad m \in \mathbb{Z},
\end{equation}
and (see (2.15) of \cite{ref:7}, or (2) of page 3 of \cite{ref:15})
\begin{equation}
  \label{eq:2-2}
  \mu(\alpha B) = |\alpha|\mu(B),
\end{equation}
for all $\alpha\in k_{\infty}$, and all Borel subsets $B$ of $k_{\infty}$. If $\alpha\in k_{\infty}$, and $r\in \mathbb{Z}$ are given, we denote a ball of $k_{\infty}$ by
$$b(\alpha, r) = \{x\in k_{\infty}: v(x-\alpha) \geq r\}.$$
$b(\alpha, r)$ is usually said to be a ball of center $\alpha$ and radius $r$. We have $\mu(B(\alpha, r)) = q^{1-r}$ (see \cite{ref:14}, pp. 65--70), and $b(\alpha, r_2) \subset b(\alpha, r_1)$ if $r_1 \leq r_2$. In particular, if $\alpha\in K$, $\beta \in K$, and $\alpha\neq \beta$, then
\begin{equation}
  \label{eq:2-3}
  b(\alpha, r_1) \cap b(\beta, r_2) = \emptyset, \quad \mbox{if } \min\{r_1, r_2\}\geq 1.
\end{equation}

There is a natural product measure $\mu_n$ on $k_{\infty}^n$ defined by
\begin{equation}
  \label{eq:2-4}
  \mu_n(B) = \prod_{i=1}^n \mu(B_i),
\end{equation}
for $B = B_1 \times \cdots \times B_n \subset k_{\infty}^n$.
This measure is translation invariant. Sometimes, we simply write $\mu_n = \mu$ when the context is clear.
\begin{lemma}\label{lem:2-1}
  If $B\subset k_{\infty}^n$ is an additive subgroup, and $\mu(B) = \mu_n(B)> 1$, then there is a non-zero integral point $u \in B$.
\end{lemma}

\begin{proof}
  Since $B$ is a subgroup of $k_{\infty}^n$, it suffices to show that there are $x$ and $y$ in $B$, $x\neq y$, such that $x-y$ has integral point coordinates. To prove this, if $u = (u_1, u_2, \ldots, u_n) \in K^n$ is given, let
  \begin{equation}
    \label{eq:2-5}
    \mathcal{R}_u = B \cap b(u, 1), \mbox{ and } \mathcal{R}_u' = \{x-u: x\in \mathcal{R}_u\},
  \end{equation}
where $b(u, 1) = b(u_1, 1) \times b(u_2, 1) \times \cdots \times b(u_n, 1)$. If $u, u' \in K^n$, and $u \neq u'$, it is easy to see that $\mathcal{R}_u \cap \mathcal{R}_{u'} = \emptyset$ by \eqref{eq:2-3}. So we have
\begin{equation}
  \label{eq:2-6}
  \mu(B) = \sum_{u}\mu(\mathcal{R}_u) = \sum_{u}\mu(\mathcal{R}'_u) > 1.
\end{equation}

If $\mathcal{R}'_u \cap \mathcal{R}'_{u'} = \emptyset$ for all integral points $u$ and $u'$ with $u\neq u'$, then there is a contradiction by \eqref{eq:2-6}. Since $\mathcal{R}_u' \subset N^n$ for all integral points, and so $\bigcup\limits_{u}\mathcal{R}_u' \subset N^n$, it follows that
\begin{equation}
  \label{eq:2-7}
  \mu(\bigcup\limits_u \mathcal{R}'_u) = \sum_{u}\mu(\mathcal{R}'_u) \leq 1.
\end{equation}
Therefore, there are at least two different integral points $u$ and $u'$ such that $\mathcal{R}'_u \cap \mathcal{R}'_{u'}$ is not empty. Hence, we may find $x\in \mathcal{R}_u$, $y \in \mathcal{R}_u$, such that $x-y = u-u' \neq 0$, completing the proof.
\end{proof}

\begin{remarks}
  Letting $B=N^n$ in Lemma~\ref{lem:2-1}, it is easy to see that the condition $\mu(B)>1$ is  best possible.
\end{remarks}

The next lemma gives an integral formula for a linear change of variables. It may be deduced from  \cite[page~17]{ref:n8} or \cite[page~7, Corollary~3]{ref:15}. Here we give an elementary proof using induction.

\begin{lemma}\label{lem:2-2}
  For $n\geq 1$, $f\in L^1(\mu_n)$, and $u\in GL(n, k_{\infty})$, one has
  \begin{equation}
    \label{eq:2-8}
    \int_{k_{\infty}^n}f(ux){\rm d}\mu_n(x) = \frac{1}{|\det u|}\int_{k_{\infty}^n} f(x){\rm d}\mu_n(x).
  \end{equation}
\end{lemma}

\begin{proof}
  The proof goes by induction on $n$. The compatibility condition \eqref{eq:2-2} ensures that Formula \eqref{eq:2-8} holds when $n=1$.

Assume now that this formula holds for a certain $n$, for all $u$ and $f$. Take $u\in GL(n+1, k_{\infty})$. If the coordinates are suitably enumerated, the matrix of $u$ takes the form
\begin{equation*}
  {\mathsf M} = \left(
    \begin{array}[c]{cc}
      {\mathsf A} & {\mathsf C}\\
{\mathsf L} & a
    \end{array}
\right),
\end{equation*}
where ${\mathsf A}$ is an invertible $n\times n$-matrix, ${\mathsf C}$ a column matrix, ${\mathsf L}$ a row matrix, and $a\in k_{\infty}$. The following formula is easily checked:
\begin{equation}
\label{eq:2-9}
  {\mathsf M} = \left(
    \begin{array}[c]{cc}
      {\mathsf A} & 0\\
      {\mathsf L} & 1
    \end{array}
\right) \left(
  \begin{array}[c]{cc}
    {\mathsf I}_n & {\mathsf A}^{-1}{\mathsf C}\\
    0 & a-{\mathsf LA}^{-1}{\mathsf C}
  \end{array}
\right).
\end{equation}
In particular we note from the formula that $a - {\mathsf LA}^{-1}{\mathsf C} \neq 0$.

Write points of $k_{\infty}^{n+1}$ as $({\mathsf X}, y)$, where ${\mathsf X}$ is an $n\times 1$-matrix and $y\in k_{\infty}$. Let $\delta$ and $w$ stand for the elements of $SL(n+1, k_{\infty})$, the matrices of which are the factors on the right hand side of \eqref{eq:2-9}.

Let $f\in L^1(k_{\infty}^{n+1})$. One has
\begin{equation*}
  \begin{split}
    \int_{k_{\infty}^{n+1}} f \circ \delta~ {\rm d}\mu_{n+1} &= \int_{k_{\infty}^n}\left(\int f({\mathsf AX}, y+{\mathsf LX}){\rm d}y \right){\rm d}{\mathsf X}\\
&= \iint_{k_{\infty}^{n+1}}f({\mathsf AX}, y)~{\rm d}{\mathsf X}{\rm d}y \mbox{~~~~~~~~(translation invariance of $\mu$)} \\
&= |\det {\mathsf A}|^{-1} \iint_{k_{\infty}^{n+1}}f({\mathsf X}, y)~{\rm d}{\mathsf X}{\rm d}y \mbox{~~~~(recursion hypothesis)}\\
&= |\det \delta|^{-1}\int_{k_{\infty}^{n+1}}f~{\rm d}\mu_{n+1}.
  \end{split}
\end{equation*}

In the same way one has, setting $b = a - {\mathsf LA}^{-1}{\mathsf C}$,
\begin{equation*}
\begin{split}
    \int_{k_{\infty}^{n+1}} f \circ w~ {\rm d}\mu_{n+1} &= \int_{k_{\infty}}\left(\int_{k_\infty^n} f({\mathsf X}+{\mathsf A}^{-1}{\mathsf C}y, by){\rm d}{\mathsf X} \right){\rm d}y\\
&= \iint_{k_{\infty}^{n+1}}f({\mathsf X}, by){\rm d}{\mathsf X}{\rm d}y \mbox{~~~~~~~~(translation invariance of $\mu_n$)} \\
&= |b|^{-1} \iint_{k_{\infty}^{n+1}}f({\mathsf X}, y){\rm d}{\mathsf X}{\rm d}y \mbox{~~~~(scalar case)}\\
&= |\det w|^{-1}\int_{k_{\infty}^{n+1}}f{\rm d}\mu_{n+1}.
\end{split}
\end{equation*}

Finally,
\begin{equation*}
  \begin{split}
    \int_{k_{\infty}^{n+1}} f \circ w~ {\rm d}\mu_{n+1} &= \int_{k_\infty^{n+1}} f\circ \delta \circ w~{\rm d}\mu_{n+1}\\
&= |\det w|^{-1}\int_{k_{\infty}^{n+1}}f\circ \delta~{\rm d}\mu_{n+1}\\
&= |\det w|^{-1}|\det \delta|^{-1}\int_{k_{\infty}^{n+1}}f~{\rm d}\mu_{n+1}\\
&= |\det u|^{-1}\int_{k_{\infty}^{n+1}}f~{\rm d}\mu_{n+1}.
  \end{split}
\end{equation*}

\end{proof}

Next, we construct some additive subgroups of $k_{\infty}^n$ by making use of some linear type inequalities. Let
\begin{equation*}
  L(x) = L(x_1, \ldots, x_m) = \sum_{i=1}^m a_ix_i, \quad a_i \in k_{\infty},
\end{equation*}
    a linear form in $m$ variables. If ${\mathsf A} = (a_{ij})_{n\times m}$ is an $n\times m$ matrix in $k_{\infty}$, then ${\mathsf A}$ defines $n$ linear forms $L_1(x), L_2(x), \ldots, L_n(x)$ (each in $m$ variables) by
\begin{equation*}
  \left(
    \begin{array}[c]{c}
      L_1(x)\\
      L_2(x)\\
      \vdots\\
      L_n(x)
    \end{array}
  \right) = {\mathsf A} \left(
    \begin{array}[c]{c}
      x_1\\
      x_2\\
      \vdots\\
      x_m
    \end{array}
\right),
\end{equation*}
where $L_i(x) = \sum_{j=1}^m a_{ij}x_j~(1\leq i \leq n)$. Let $r = (r_1, r_2, \ldots, r_n) \in \mathbb{Z}^n$ be $n$ given integers. Then ${\mathsf A}$ and $r$ define a subset of $k_{\infty}^m$ by
\begin{equation}
  \label{eq:2-10}
\Gamma({\mathsf A}, r) = \{x=(x_1, \ldots, x_m)\in k_{\infty}^m: v(L_i(x))\geq r_i, 1\leq i \leq n  \}.
\end{equation}
${\mathsf A}$ is said to be the corresponding matrix for the system of $n$ linear inequalities $v(L_i(x)) \geq r_i~(1\leq i \leq n)$. We have

\begin{lemma}\label{lem:2-3}
Let ${\mathsf A} = (a_{ij})_{n\times m}$ be an $n\times m$ matrix in $k_{\infty}$, and $r=(r_1,r_2, \ldots, r_n)$ be $n$ integers, then $\Gamma(A, r)$ is an additive subgroup of $k_{\infty}^m$.
\end{lemma}

\begin{proof}
  If $x = (x_1,\ldots, x_m) \in \Gamma({\mathsf A}, r)$, and $y=(y_1, y_2, \ldots, y_m) \in \Gamma({\mathsf A}, r)$, then for each $i$, $1\leq i \leq n$, we have
  \begin{equation*}
    v(\sum_{j=1}^m a_{ij}(x_j + y_j)) \geq \min\left\{v(\sum_{j=1}^m a_{ij}x_j), v(\sum_{j=1}^m a_{ij}y_j)\right\} \geq r_i.
  \end{equation*}
If follows that $x+y \in \Gamma({\mathsf A}, r)$. Suppose that $a\in \fnum_q^*$ is a non-zero constant in $\fnum_q$, then
\begin{equation*}
  v(\sum_{j=1}^n a_{ij} a x_j) = v(\sum_{j=1}^n a_{ij}x_j) \geq r_i.
\end{equation*}
Hence $ax\in \Gamma({\mathsf A}, r)$. In particular, $-x \in \Gamma({\mathsf A}, r)$. Therefore, $\Gamma({\mathsf A}, r)$ is an additive subgroup of $k_{\infty}^m$, completing the proof of the lemma.
\end{proof}

\begin{lemma}[Basic lemma]
  \label{lem:2-4} Let ${\mathsf A}=(a_{ij})_{n\times n}$ be an $n\times n$ matrix in $k_{\infty}$ with $\det({\mathsf A}) \neq 0$, and $r=(r_1, r_2, \ldots, r_n)$ be an $n$-tuple of integers. If $r_1 + r_2 + \cdots + r_n < v(\det({\mathsf A})) + n$, then there is a non-zero integral point $u \in \Gamma({\mathsf A}, r)$.
\end{lemma}

\begin{proof}
  Let ${\mathsf A}^{-1} = (b_{ij})_{n\times n}$. We write
  \begin{equation*}
    y_i = \sum_{j=1}^n a_{ij}x_j, \mbox{   and    } x_i = \sum_{j=1}^n b_{ij}y_i, ~(1 \leq i \leq n).
  \end{equation*}
If $x= (x_1, x_2, \ldots, x_n) \in \Gamma({\mathsf A}, r)$, then $v(y_i) \geq r_i$ for all $1\leq i \leq n$. By Lemma~\ref{eq:2-3}, if we write ${\rm d}u_n =  {\rm d}x_1{\rm d}x_2\cdots {\rm d}x_n$, then
\begin{equation*}
  \begin{split}
    \mu(\Gamma({\mathsf A}, r)) &= \mu_n(\Gamma({\mathsf A}, r))  = \int\cdots \int{\rm d}x_1{\rm d}x_2\cdots {\rm d}x_n\\
                      &= \int\cdots \int_{v(y_i)\geq r_i} |\det({\mathsf A}^{-1})|{\rm d}y_1{\rm d}y_2\cdots {\rm d}y_n\\
                      &= |\det({\mathsf A})|^{-1} \prod_{i=1}^n q^{1-r_i}\\
                      &= q^{v(\det({\mathsf A})) + n - \sum r_i}.
  \end{split}
\end{equation*}
If $\sum_{i=1}^n r_i < v(\det({\mathsf A})) + n$, then $\mu(\Gamma({\mathsf A}, r)) > 1$. By Lemma~\ref{lem:2-1} and Lemma~\ref{lem:2-3}, the lemma follows immediately.
\end{proof}

\begin{remarks}
  The above lemma is an analogue of Minkowski's linear forms theorem; see Theorem~III of Appendix B of \cite{ref:6}.

The above results provide basic information about the geometry of polynomials in function fields, and give us a new tool to study Diophantine approximation in the case of positive characteristic.
\end{remarks}

\begin{theorem}\label{thm:2-1}
  Suppose that $\theta \in k_{\infty}$, and $H \in k_{\infty}$ with $|H| \geq q$. Then there is a polynomial $x\in K$, such that
  \begin{equation}
    \label{eq:2-11}
    ||x\theta || \leq \frac{1}{|H|}, \mbox{   and   } 1 \leq |x| < |H|.
  \end{equation}
\end{theorem}

\begin{proof}
  We first suppose that $H\in K$ is a polynomial, and consider the following linear inequalities
  \begin{equation}
    \label{eq:2-12}
    \left\{
      \begin{tabular}[l]{l}
        $v(x\theta + y) \geq \deg H$\\
        $v(x) \geq -\deg H +1$.
      \end{tabular}
\right.
  \end{equation}
The corresponding matrix of \eqref{eq:2-12} is ${\mathsf A} = \left(
  \begin{array}[c]{cc}
    \theta & 1\\
    1 & 0
  \end{array}
\right)$, and so $v(\det({\mathsf A})) = 0$. Since $r_1 + r_2 = 1$, by Lemma~\ref{lem:2-4}, there is a non-zero integer point $(x, y)\in K^2$ satisfying \eqref{eq:2-12}. Since $(x, y) \neq 0$, we claim that $x\neq 0$. Indeed, if $x=0$, then by \eqref{eq:2-12}, $v(y)\geq \deg H \geq 1$, but $y\in K$, $y\neq 0$, and so $v(y)\leq 0$, a contradiction. Hence $x\neq 0$, and it follows by \eqref{eq:2-12} again that
\begin{equation}
  \label{eq:2-13}
  ||x\theta|| \leq |x\theta+y| \leq \frac{1}{|H|}, \mbox{   and   } 1\leq |x| \leq \frac{1}{q}|H| < |H|,
\end{equation}
yielding the result of the theorem.
If $H$ is not a polynomial, we replace $H$ by $[H]$, and note that $|H| = |[H]|$, when $|H|\geq 1$~(see \eqref{eq:1-3}). Again the theorem holds true.
\end{proof}

\begin{remarks}
  If $H$ is a polynomial in Theorem \eqref{thm:2-1}, and $\theta = \frac{1}{H}$, it is easy to see that
  \begin{equation*}
    \left|\left|\frac{x}{H}\right|\right| = \left|\frac{x}{H}\right| \geq \frac{1}{|H|}
  \end{equation*}
for all $x\in K$ and $1\leq |x| < |H|$. Hence the ``$\leq$'' in \eqref{eq:2-11} cannot be improved to ``$<$''. The theorem above is clearly well known as Dirichlet's theorem, but the proof we give here is quite different from the classical method, for example, the proof of Theorem~9.1.1 of \cite{ref:n3}.
\end{remarks}

As we already pointed out in the introduction, if $|x|<|H|$, then $|x|\leq \frac{1}{q}|H|$. By \eqref{eq:2-11}, it follows that
\begin{equation*}
  |x|~||x\theta|| \leq \frac{1}{q}|H|~||x\theta|| \leq \frac{1}{q}.
\end{equation*}
Thus the following corollary is immediate.

\begin{corollary}\label{cor:2-1}
  If $\theta \in k_{\infty}$, then there are infinitely many polynomials $x\in K$, such that
  \begin{equation}
    \label{eq:2-14}
    |x|~||x\theta||\leq \frac{1}{q}.
  \end{equation}
\end{corollary}

This corollary implies that the approximation constant $\tau(\alpha)$ given by \eqref{eq:1-12} is bounded by $\frac{1}{q}$ for all $\alpha\in k_{\infty}$, and that the upper bound $\frac{1}{q}$ is also best possible.

Next we give a proof of Theorem~\ref{thm:1-1}, which just is a high dimensional version of the preceding theorem.

\begin{proof}[Proof of Thoerem~\ref{thm:1-1}]
  Suppose that $\theta_1, \theta_2, \ldots, \theta_n$ are $n$ given elements in $k_{\infty}$, and $H\in k_{\infty}$ with $|H|\geq 1$. Without loss of generality, we suppose $H\in K$ is a polynomial with $\deg H \geq 0$, and consider the following $n+1$ linear inequalities
  \begin{equation}
    \label{eq:2-15}
    \left\{
      \begin{tabular}[l]{l}
        $v(\theta_ix+y_i) \geq \left[\frac{\deg H}{n}\right]+1, 1\leq i \leq n,$\\
        $v(x) \geq -\deg(H).$
      \end{tabular}
    \right.
  \end{equation}
It is easy to see that $v(\det(A)) = 0$, where ${\mathsf A}$ is the corresponding matrix of \eqref{eq:2-15}, and
\begin{equation*}
  r_1 + r_2 + \cdots + r_{n+1} = n + n\left[\frac{\deg H}{n}\right] - \deg H \leq n.
\end{equation*}

By Lemma~\ref{lem:2-4}, there is a non-zero integral point $X = (x, y_1, y_2, \ldots, y_n)$ such that \eqref{eq:2-15} holds. If $x=0$, then there is a $y_i\neq 0$, and by \eqref{eq:2-15},
\begin{equation*}
  v(y_i)\geq \left[\frac{\deg H}{n}\right] + 1 \geq 1,
\end{equation*}
which contradicts $v(y_i) \leq 0$. Hence we have $x \neq 0$, and
\begin{equation*}
  || x \theta_i ||^n \leq |x\theta_i + y_i|^n < |H|^{-1}, 1\leq i \leq n, \mbox{ and } 1\leq |x| \leq |H|.
\end{equation*}
It follows that
$$||x\theta_i||^n \leq (q|H|)^{-1}, 1\leq i \leq n, $$
and we thus have Theorem~\ref{eq:1-1} immediately.
\end{proof}

\begin{corollary}\label{cor:2-2}
  If $\theta_1, \theta_2, \ldots, \theta_n\in k_{\infty}$, and at least one of $\theta_i \not \in k$, then there are infinitely many polynomials $x\in K$, such that
  \begin{equation}
    \label{eq:2-16}
    |x|^{\frac{1}{n}} \max\limits_{1\leq i \leq n} ||x\theta_i|| \leq \frac{1}{q}.
  \end{equation}
\end{corollary}

\begin{proof}
  We take $\deg H$ as a multiple of $n$ in Theorem~\ref{thm:1-1}, so that $\frac{\deg H}{n}$ is an integer. By \eqref{eq:2-15}, we have
  \begin{equation*}
    ||x\theta_i|| \leq |x\theta_i + y_i| \leq \frac{1}{q}|H|^{-\frac{1}{n}}
  \end{equation*}
and $1 \leq |x| \leq |H|$. It follows that
\begin{equation*}
  |x|^{\frac{1}{n}} ||x\theta_i|| \leq |H|^{\frac{1}{n}}||x\theta_i|| \leq \frac{1}{q}, \quad 1 \leq i \leq n.
\end{equation*}
The corollary follows by letting $|H| \rightarrow \infty$.
\end{proof}

\begin{remarks}
  In the classical case, the simultaneous approximation constant $\frac{1}{q}$ in \eqref{eq:2-16} is replaced by $\frac{n}{n+1}$ (see Theorem~VII, page 14 of \cite{ref:6}), a value that tends to $1$  as $n\to \infty$. The situation is quite different from the positive characteristic case.
\end{remarks}

There is a transposed form of Theorem~\ref{thm:1-1}, which we may state as follows.

\begin{corollary}\label{cor:2-3}
  If $\theta_1, \theta_2, \ldots, \theta_n$ are $n$ elements in $k_{\infty}$, and $H\in k_{\infty}$ with $|H|\geq 1$, then there is an integral point $x = (x_1, x_2, \ldots, x_n)$, such that
  \begin{equation}
    \label{eq:2-17}
    ||x_1\theta_1 + x_2\theta_2 + \cdots + x_n\theta_n || \leq \frac{1}{q^n} |H|^{-n}, \mbox{ and } 1 \leq \max\limits_{1\leq i \leq n}|x_i| \leq |H|.
  \end{equation}
\end{corollary}

\begin{proof}
  Let $H\in K$ be a polynomial, and consider the following $n+1$ linear inequalities
\begin{equation}
    \label{eq:2-18}
    \left\{
      \begin{tabular}[l]{l}
        $v(\theta_1x_1+\cdots + \theta_nx_n+y) \geq n(\deg H+1)$\\
        $v(x_i) \geq -\deg(H), 1\leq i \leq n.$
      \end{tabular}
    \right.
  \end{equation}
We have $v(\det({\mathsf A})) = 0$, and $r_1+r_2+\cdots+r_{n+1} = n$. By Lemma~\ref{lem:2-4}, there is a non-zero integral point $x= (x_1, x_2, \ldots, x_n)$, such that
\begin{equation*}
   ||x_1\theta_1 + \cdots + x_n\theta_n || \leq q^{-n} |H|^{-n}, \mbox{ and } 1 \leq \max\limits_{1\leq i \leq n}|x_i| \leq |H|.
\end{equation*}
The corollary follows at once.
\end{proof}

\begin{remarks}
  By \cite[Theorem~1.1]{ref:nn2}, one has
  \begin{equation*}
    ||x_1\theta_1 + x_2\theta_2 + \cdots + x_n\theta_n || \leq \frac{1}{q} |H|^{-n}, \mbox{ and } 1 \leq \max\limits_{1\leq i \leq n}|x_i| \leq |H|,
  \end{equation*}
which is weaker than Corollary~\ref{cor:2-3} as $n > 1$.
\end{remarks}

As a direct consequence of Corollary~\ref{cor:2-3}, we have
\begin{corollary}
  Let $\theta_1, \theta_2, \ldots, \theta_n$ be $n$ elements in $k_{\infty}$. Then there are infinitely many integral points $x = (x_1, x_2, \ldots, x_n)\in K^n$, such that
  \begin{equation}
    \label{eq:2-19}
    (\max\limits_{1\leq i \leq n}|x_i|)^n||x_1\theta_1 + \cdots + x_n\theta_n || \leq \frac{1}{q^n}.
  \end{equation}
\end{corollary}

The main result of this section is the following more generalized linear forms approximation theorem, which is an anlogue of Theorem~VI of Chapter I of \cite{ref:6}.

\begin{theorem}
  \label{thm:2-2} Let
  \begin{equation*}
    L_i(x) = \sum_{j=1}^m \theta_{ij}x_j, \quad 1\leq i \leq n, \quad \theta_{ij}\in k_{\infty},
  \end{equation*}
be $n$ linear forms in $m$ variables ($x=(x_1,\ldots, x_m)$). Then for any $H\in k_{\infty}$, $|H|\geq 1$, there is a non-zero integral point $x=(x_1,\ldots,x_m)\in K^m$, such that
\begin{equation}
  \label{eq:2-20}
  \max\limits_{1\leq i \leq n}||L_i(x)|| \leq q^{-\frac{m}{n}} \cdot |H|^{-\frac{m}{n}}, \mbox{ and } 1\leq \max\limits_{1\leq j\leq m}|x_j| \leq |H|.
\end{equation}
\end{theorem}

\begin{proof}
  Without loss of generality, we may suppose $H\in K$, and consider the following $n+m$ linear inequalities,
\begin{equation}
    \label{eq:2-22}
    \left\{
      \begin{tabular}[l]{l}
        $v(L_i(x)+y_i) \geq [\frac{m\deg H + m - 1}{n}]+1, 1\leq i \leq n,$\\
        $v(x_j) \geq -\deg H, 1\leq j \leq m.$
      \end{tabular}
    \right.
  \end{equation}
It is easy to see that $v(\det({\mathsf A})) = 0$, where ${\mathsf A}$ is the corresponding matrix of \eqref{eq:2-22}, and
\begin{equation*}
  \sum_{i=1}^{n+m} r_i = -m\deg H + n\left[\frac{m\deg H + m - 1}{n}\right] + n \leq n + m -1.
\end{equation*}
By Lemma~\ref{lem:2-4}, there is a non-zero integral point $(x, y)\in K^{n+m}$ satisfying \eqref{eq:2-22}. If $x=(x_1, x_2, \ldots, x_m) = 0$, then there is at least one $y_i$, such that $y_i\in K$ and $y_i\neq 0$. Then by \eqref{eq:2-22} we have $v(y_i)\geq 1$, which is a contradiction to $v(y_i)\leq 0$. Thus we have an integral point $x\neq 0$, such that
\begin{equation}
  \label{eq:2-22-r}
  v(L_i(x) + y_i)> \frac{m\deg H + m - 1}{n}.
\end{equation}
It follows that
\begin{equation*}
  ||L_i(x)|| \leq q^{-\frac{m}{n}}|H|^{-\frac{m}{n}}, 1\leq i \leq n
\end{equation*}
and
\begin{equation*}
  1 \leq \max\limits_{1\leq j \leq m} |x_j| \leq |H|,
\end{equation*}
completing the proof of Theorem~\ref{thm:2-2}.
\end{proof}

\begin{remarks}
If $m=1$ in \eqref{eq:2-20}, then Theorem~\ref{thm:2-2} becomes Theorem~\ref{thm:1-1}. If $n=1$ in \eqref{eq:2-20}, then Theorem~\ref{thm:2-2} becomes Corollary~\ref{cor:2-3}. In \cite{ref:nn3}, the author showed (see \cite[Theorem~1.1]{ref:nn3}) that
  \begin{equation*}
  \max\limits_{1\leq i \leq n}||L_i(x)|| \leq |H|^{-\frac{m}{n}}, \mbox{ and } 1\leq \max\limits_{1\leq j\leq m}|x_j| \leq |H|
  \end{equation*}
under the condition of $n \,|\, m\cdot \deg(H)$. In \cite{ref:nn2}, the authors proved the following estimation (see \cite[Theorem~2.1]{ref:nn2})
\begin{equation*}
  \prod_{j=1}^m |x_j| \cdot \prod_{i=1}^n ||L_i(x)|| \leq q^{-m}.
\end{equation*}
Clearly, both the above results are weaker than our Theorem~\ref{thm:2-2}.
\end{remarks}

To compare Theorem~\ref{thm:2-1} of \cite{ref:nn2} in detail, let $\mathbb{Z}_+^m$ denote the set of all $m$ tuples $(t_1, t_2, \ldots, t_m)$ where each $t_i$ is a nonnegative integer, then Theorem~\ref{thm:2-2} above may be rewritten in the following form.

\begin{corollary}
  \label{cor:2-5}
Let
  \begin{equation*}
    L_i(x) = \sum_{j=1}^m \theta_{ij}x_j, \quad 1\leq i \leq n, \quad \theta_{ij}\in k_{\infty}
  \end{equation*}
be $n$ linear forms in $m$ variables $x=(x_1,\ldots, x_m)$. Then for any $t = (t_1, \ldots, t_{n+m}) \in \mathbb{Z}_+^{n+m}$ with $\sum_{i=1}^n t_i  = \sum_{j=1}^m t_{n+j}$, there is a non-zero integral point $x=(x_1, \ldots, x_m)\in K^m$ such that
\begin{equation*}
  \left\{
    \begin{tabular}[l]{ll}
      $||L_i(x)|| \leq q^{-t_i-1-\delta_i}$,  & $1\leq i \leq n$\\
      $|x_j| \leq q^{t_{n+j}}$, & $1\leq j \leq m$
    \end{tabular}
\right.
\end{equation*}
for all nonnegative integers $\delta_i~(1\leq i \leq n)$ such that $\sum_{i=1}^n \delta_i \leq m-1$.
\end{corollary}

\begin{proof}
  We consider the following $n+m$ linear inequalities
\begin{equation*}
  \left\{
    \begin{tabular}[l]{ll}
      $v(L_i(x) + y_i) \geq 1+t_i+\delta_i$,  & $1\leq i \leq n$\\
      $v(x_j) \geq -t_{n+j}$, & $1\leq j \leq m$
    \end{tabular}
\right.
\end{equation*}
It is easy to see that $v(det(A)) = 0$, where $A$ is its corresponding matrix, and
\begin{equation*}
  \sum_{i=1}^{n+m} r_i = n+ \sum_{i=1}^n \delta_i \leq n+m-1.
\end{equation*}
By Lemma~\ref{lem:2-4}, the assertion follows immediately.
\end{proof}

Comparing Corollary~\ref{cor:2-5} with Theorem~\ref{thm:2-1} of \cite{ref:nn2}, we save totally $m-1$ factors of $q^{-1}$. Theorem~\ref{thm:2-2} is best possible in a certain sense as the following theorem shows.

\begin{theorem}
  \label{thm:2-3} If $n$ and $m$ are two positive integers such that $p \nmid n+m$, and there exist $n+m$ conjugate integral algebraic elements of degree $n+m$ in $k_{\infty}$, then there is a constant $r>0$, and $n$ linear forms $L_i(x)~(1\leq i \leq n)$,  each in $m$ variables, such that
  \begin{equation}
    \label{eq:2-24}
    (\max\limits_{1\leq j \leq m}|x_j|)^m (\max\limits_{1\leq i \leq n}||L_i(x)||)^n \geq r,
  \end{equation}
for all $x=(x_1, x_2, \ldots, x_m)\in K^m$ and $x\neq 0$.
\end{theorem}

\begin{proof}
  Let $l=n+m$. Since $p \nmid l$, there are $l$ conjugate algebraic integral elements $\alpha_1, \alpha_2, \ldots, \alpha_l$ of degree $l$ with $\alpha_i \neq \alpha_j$ (if $i\neq j$). We put
  \begin{equation}
    \label{eq:2-25}
    Q_s(x,y) = \sum_{i=1}^n \alpha_s^{i-1}y_i + \sum_{j=1}^m \alpha_s^{n+j-1}x_j, \quad 1\leq s \leq l.
  \end{equation}
If integral points $x=(x_1,x_2, \ldots, x_m)$, and $y=(y_1,y_2,\ldots,y_n)$ not both are zero,
then
$Q_s(x, y)$ $~(1\leq s \leq l)$ are conjugate non-zero integral elements. In particular, $\prod_{s=1}^l Q_s(x,y)$ is then a non-zero polynomial, and so
\begin{equation}
  \label{eq:2-26}
  \prod_{s=1}^l |Q_s(x, y)| \geq 1.
\end{equation}

The matrix ${\mathsf A} = (\alpha_s^{r-1})~(1\leq s \leq n, 1\leq r\leq n)$ is of Vandermonde type, and so its determinant is not zero. By \eqref{eq:2-25}, we may solve $y$ in terms of $x$ by equations $Q_1=Q_2=\cdots = Q_n=0$, and obtain $y_1=L_1(x), y_2=L_2(x), \ldots, y_n=L_n(x)$. It follows that
\begin{equation*}
  \sum_{i=1}^n \alpha_s^{i-1} L_i(x) + \sum_{j=1}^m \alpha_s^{n+j-1}x_j = 0, \quad 1\leq s\leq n
\end{equation*}
and
\begin{equation}
  \label{eq:2-27}
  Q_s(x,y) = \sum_{i=1}^n \alpha_s^{i-1}(y_i - L_i(x)), \quad 1\leq s\leq n.
\end{equation}
For $s>n$, we have
\begin{equation}
  \label{eq:2-28}
  Q_s(x,y) = \sum_{i=1}^n \alpha_s^{i-1}(y_i - L_i(x)) + \sum_{j=1}^m \beta_{sj}x_j,
\end{equation}
where $\beta_{sj}$ depend only on $\alpha_1, \ldots, \alpha_l$. If $x\neq 0$ is an integral point, we put
\begin{equation}
  \label{eq:2-29}
  A = \max\limits_{1\leq j \leq m}|x_j|,\mbox{ and } C=\max\limits_{1\leq i \leq n}||L_i(x)||.
\end{equation}

Let $y=(y_1,y_2, \ldots, y_n)\in K^n$, such that
\begin{equation*}
  ||L_i(x) || = |L_i(x) - y_i|,
\end{equation*}
in fact, where $y_i = [L_i(x)]~(1\leq i \leq n)$. Then by \eqref{eq:2-27}, we have
\begin{equation}
  \label{eq:2-30}
  |Q_s(x, y)| \leq r_1C, \quad 1\leq s\leq n,
\end{equation}
where $r_1$ depends only on the $\alpha_s$ but not on $x$. Similarly, by \eqref{eq:2-28},
\begin{equation}
  \label{eq:2-31}
  |Q_s(x, y)| \leq r_2C + r_3A \leq r_4A, \quad n\leq s\leq l,
\end{equation}
where $r_2, r_3$ and $r_4$ depend only on the $\alpha_s$ but not on $x$. From \eqref{eq:2-26}, we have
\begin{equation}
  \label{eq:2-32}
  1 \leq \prod_{s=1}^l |Q_s(x, y)| \leq r_1^n C^n r_4^m A^m.
\end{equation}
Now \eqref{eq:2-24} follows with $r=r_1^{-n} r_4^{-m}$. This completes the proof of Theorem~\ref{thm:2-3}.
\end{proof}

\section{Transference Principle}
\label{sec:trans}

Let
$$L_i(x) = \sum_{j=1}^m a_{ij}x_j, \quad 1\leq i \leq n$$
be a set of $n$ linear forms with a corresponding matrix ${\mathsf A} = (a_{ij})_{n\times m}$. Let ${\mathsf A}' = (a_{ji})_{m\times n}$ be the transpose of $\mathsf A$ and the transposed linear forms of the $L_i(x)$ be defined by
$$M_j(y) = \sum_{i=1}^n a_{ji}y_i, \quad 1\leq j \leq m.$$
It is interesting to note that the $L_i(x)$ have a strong tie with the $M_j(y)$ in Diophantine approximation. This relationship is called the transference principle. We start by proving the following lemma.

\begin{lemma}
  \label{lem:3-1} Let $f_i(z)~(1\leq i \leq n)$ be $n$ linearly independent linear forms in $n$ variables $z=(z_1, \ldots, z_n)$, and $g_i(w)~(1\leq i \leq n)$ be also $n$ linearly independent linear forms in $n$ variables $w = (w_1, w_2, \ldots, w_n)$ of determinant $d~(d\in k_{\infty})$. Let $n\geq 2$, and
  \begin{equation}
    \label{eq:3-1}
    \psi(z, w) = \sum_{i=1}^n f_i(z)g_i(w).
  \end{equation}
Suppose that all of the products $z_iw_j~(1\leq i, j \leq n)$ have polynomial coefficients in \eqref{eq:3-1}. If the inequalities
\begin{equation}
  \label{eq:3-2}
  |f_i(z)| \leq |\lambda|, \quad (1\leq i \leq n)
\end{equation}
are solvable with an integral point $z\neq 0$, then the inequalities
\begin{equation}
  \label{eq:3-3}
  |g_i(w)| \leq q^{\frac{2-n}{n-1}}|\lambda\cdot d|^{\frac{1}{n-1}}
\end{equation}
are solvable with an integral point $w\neq 0$.
\end{lemma}

\begin{proof}
  Let $|\lambda| = q^{-m}$. Then $|f_i(z)| \leq |\lambda|$, if and only if $v(f_i(z)) \geq m~(1\leq i \leq n)$. We may suppose that $v(f_n(z)) = m$, and $z=(z_1,z_2,\ldots, z_n)\neq 0$ is the corresponding integral point. Next, we consider the following $n$ linear inequalities in $n$ variables $w=(w_1,w_2, \ldots, w_n)$
  \begin{equation}
    \label{eq:3-4}
    \left\{
      \begin{tabular}[l]{l}
        $v(\psi(z, w)) \geq 1$\\
        $v(g_i(w)) \geq [\frac{v(d)+n+m-2}{n-1}], 1\leq i \leq n-1.$
      \end{tabular}
\right.
  \end{equation}

It is easy to verify that the determinant of the corresponding matrix of \eqref{eq:3-4} is $f_n(z)\cdot d$, and $v(f_n(z)d) = m+v(d)$. From \eqref{eq:3-4} we have
\begin{equation*}
  \sum_{i=1}^n r_i = 1 + (n-1)\left[\frac{v(d)+n+m-2}{n-1}\right] \leq v(f_n(z)\cdot d)+n-1.
\end{equation*}
By Lemma~\ref{eq:2-4}, there is a non-zero integral point $w=(w_1, \ldots, w_n)\in K^n$, such that \eqref{eq:3-4} holds. Since $\psi(z, w)\in K$, by \eqref{eq:3-4} we have $\psi(z, w) = 0$. By \eqref{eq:3-1}, we have
\begin{equation*}
  \begin{split}
    m + v(g_n(w)) &= v(-\sum_{i=1}^{n-1}f_i(z) g_i(w)) \\
 & \geq \min\limits_{1\leq i \leq n-1}\left\{v(f_i(z)) + v(g_i(w))  \right\}\\
 & \geq m+\left[ \frac{v(d)+n+m-2}{n-1} \right].
  \end{split}
\end{equation*}
It follows from \eqref{eq:3-4} that
\begin{equation*}
  v(g_i(w)) \geq \left[\frac{v(d)+m+n-2}{n-1} \right]
\end{equation*}
for all $i$ with $1\leq i \leq n$. Therefore, we have
\begin{equation}
  \label{eq:3-5}
  |g_i(w)| \leq q^{\frac{2-n}{n-1}}|\lambda d|^{\frac{1}{n-1}}, \quad 1\leq i \leq n,
\end{equation}
with $w=(w_1,w_2, \ldots, w_n)\neq 0$, completing the proof of Lemma~\ref{lem:3-1}.
\end{proof}

A direct application of the above lemma gives the following results.

\begin{theorem}
  \label{thm:3-1} Let ${\mathsf A} = (a_{ij})_{n\times m}$ be a matrix in $k_{\infty}$, and
  \begin{equation*}
    L_i(x) = \sum_{j=1}^m a_{ij}x_j, \mbox{ and } M_j(y) = \sum_{i=1}^n a_{ji}y_i.
  \end{equation*}
If there is an integral point $x\neq 0$, such that
\begin{equation}
  \label{eq:3-6}
  ||L_i(x)|| \leq |C|, \mbox{ and } |x_j| \leq |X|, \quad (1\leq i \leq n, 1\leq j \leq m),
\end{equation}
where $0<|C|\leq \frac{1}{q}$, and $1\leq |X|$, then there is an integral point $y\neq 0$, such that
\begin{equation}
  \label{eq:3-7}
  ||M_j(y)|| \leq |D|, \mbox{ and } |y_i| \leq |Y|, \quad (1\leq i \leq n, 1\leq j \leq m),
\end{equation}
where $(l=n+m)$,
\begin{equation}
  \label{eq:3-8}
  |D| = q^{\frac{2-l}{l-1}}|X|^{\frac{1-n}{l-1}}|C|^{\frac{n}{l-1}}, \quad |Y|=q^{\frac{2-l}{l-1}}|X|^{\frac{m}{l-1}}|C|^{\frac{1-m}{l-1}}.
\end{equation}
\end{theorem}

\begin{proof}
  We introduce new variables
$$ w= (w_1, w_2, \ldots, w_n), \mbox{ and } u = (u_1,u_2, \ldots, u_m).$$
Put
\begin{equation}
\label{eq:3-9}
  f_i(x, w) = \left\{
    \begin{tabular}[l]{l}
      $C^{-1}(L_i(x) + w_i)$, $1\leq i \leq n$\\
      $X^{-1}x_{i-n}$, $n < i \leq l,$
    \end{tabular}\right.
\end{equation}
and
\begin{equation}
  \label{eq:3-10}
  g_i(y, u) = \left\{
    \begin{tabular}[l]{l}
      $Cy_i$, $1\leq i \leq n$\\
      $X(u_{i-n} - M_{i-n}(y))$, $n < i \leq l.$
    \end{tabular}\right.
\end{equation}
The corresponding matrix of \eqref{eq:3-10} is an $l\times l$ matrix in $k_{\infty}$, and its determinant is $d=C^nX^m$. Moreover, it is easy to see that
\begin{equation}
  \label{eq:3-11}
  \sum_{i=1}^l f_i(x, w)g_i(y, u) = \sum_{i=1}^n y_iw_i + \sum_{j=1}^m x_j u_j,
\end{equation}
since the terms in $x_jy_i$ all cancel out. By \eqref{eq:3-6}, there is a non-zero integral point $(x, w)$ such that
$$|f_i(x, w)|\leq 1, \quad 1\leq i \leq l.$$
Thus, we may apply Lemma~\ref{lem:3-1} with $|\lambda| = 1$. Hence there is a non-zero integral point $(y, u)$, such that
\begin{equation}
  \label{eq:3-12}
  |X|~|M_j(y) - u_j| \leq q^{\frac{2-l}{l-1}}|C^n X^m|^{\frac{1}{l-1}}, \quad 1\leq j\leq m,
\end{equation}
and
\begin{equation}
  \label{eq:3-13}
  |Cy_i|\leq q^{\frac{2-l}{l-1}}|C^nX^m|^{\frac{1}{l-1}}, \quad 1\leq i \leq n.
\end{equation}
Since $0 < |C| \leq \frac{1}{q}$, and $|X|\geq 1$, we have $|D| < 1$. If $y = 0$, by \eqref{eq:3-12} we have $u=0$, which is a contradiction. Therefore, $y$ is a non-zero integral point such that
\begin{equation*}
  ||M_j(y)|| \leq q^{\frac{2-l}{l-1}}|C|^{\frac{n}{l-1}}|X|^{\frac{1-n}{l-1}}, \quad 1\leq j \leq m
\end{equation*}
and
\begin{equation*}
  |y_i| \leq q^{\frac{2-l}{l-1}}|C|^{\frac{1-m}{l-1}}|X|^{\frac{m}{l-1}}, \quad 1\leq i \leq n,
\end{equation*}
completing the proof of Theorem~\ref{thm:3-1}.
\end{proof}

\begin{corollary}
  \label{cor:3-1} A necessary and sufficient condition for there to exist a constant $r\in k_{\infty}$ with $|r|>0$, such that
  \begin{equation}
    \label{eq:3-14}
    (\max\limits_{1\leq i \leq n}||L_i(x)||)^n (\max\limits_{1\leq j \leq m}|x_j|)^m \geq |r|,
  \end{equation}
for all integral points $x\neq 0$, is that there exists a constant $\delta\in k_{\infty}$ with $|\delta|>0$, such that
\begin{equation}
  \label{eq:3-15}
    (\max\limits_{1\leq j \leq m}||M_j(y)||)^m (\max\limits_{1\leq i \leq n}|y_i|)^n \geq |\delta|,
\end{equation}
for all integral points $y\neq 0$.
\end{corollary}

\begin{proof}
  For any fixed integral point $x\neq 0$, let
  \begin{equation}
    \label{eq:3-16}
    |X| = \max\limits_{1\leq j \leq m}|x_j|, \mbox{ and } |C| = \max\limits_{1\leq i \leq n}||L_i(x)||,
  \end{equation}
where $0< |C| \leq \frac{1}{q}$, and $1\leq |X|$. By Theorem~\ref{thm:3-1}, there is an integral point $y\neq 0$, such that
\begin{equation}
  \label{eq:3-17}
  \left((\max||M_j(y)||)^m (\max|y_i|)^n \right)^{l-1}\leq q^{(2-l)l}|X|^m|C|^n.
\end{equation}
If there is a $|\delta|>0$, such that \eqref{eq:3-15} holds for all integral points $y\neq 0$, it follows that
\begin{equation}
  \label{eq:3-18}
  |X|^m |C|^n \geq q^{(l-2)l}|\delta|^{l-1}.
\end{equation}
Thus \eqref{eq:3-14} holds with $|r| = q^{(l-2)l}|\delta|^{l-1}$. Similarly, the symmetry of the relation between the $L_i(x)$ and the $M_j(y)$ shows that if $r$ exists then so does $\delta$, completing the proof of Corollary~\ref{cor:3-1}.
\end{proof}

Next, we give a proof of Theorem~\ref{thm:1-2} by applying the above transference theorem.

\begin{proof}[Proof of Theorem~\ref{thm:1-2}]
  Suppose that $\theta_1, \theta_2, \ldots, \theta_n$ are elements in $k_{\infty}$ such that $k(\theta_1, \dots, \theta_n)$ is of degree $n+1$ over $k$, and $\{1, \theta_1, \ldots, \theta_n\}$ are linearly independent over $k$.
  By Theorem~\ref{thm:3-1} and Corollary~\ref{cor:3-1}, it suffices to establish \eqref{eq:1-10}, namely that there is a $|\delta|>0$, such that
  \begin{equation}
    \label{eq:3-19}
    ||u_1\theta_1 + u_2\theta_2 + \cdots + u_n\theta_n || \geq |\delta| (\max\limits_{1\leq j \leq n}|u_j|)^{-n}
  \end{equation}
for all integral points $u\neq 0$. If $u\in K^n$ and $u\neq 0$ is given, let $w= -[u_1\theta_1 + \cdots + u_n\theta_n]$. Then
\begin{equation}
  \label{eq:3-20}
  ||u_1\theta_1 + \cdots + u_n\theta_n|| = |w+u_1\theta_1 +\cdots+ u_n\theta_n|\leq \frac{1}{q}.
\end{equation}
Since $\theta_1, \theta_2, \ldots, \theta_n$ are algebraic over $k$, there is a polynomial $A\in K$, $A\neq 0$, such that $A\theta_1, A\theta_2, \ldots, A\theta_n$ are algebraic integral elements in $k_{\infty}$. Let
\begin{equation}
  \label{eq:3-21}
  \alpha = Aw + Au_1\theta_1 + \cdots + Au_n\theta_n.
\end{equation}
Then $\alpha$ is also an algebraic integral element in $k(\theta_1, \ldots, \theta_n)$ and $\alpha \neq 0$, because $\{1, \theta_1, \ldots, \theta_n\}$ are linearly independent over $k$. Let $m$ be the degree of $\alpha$ over $k$, and $m=p^t m_0$, where $m_0$ is the degree of separability of $\alpha$ and $p^t$ is the degree of inseparability of $\alpha$ (see Definition~2 of Page~67 of \cite{ref:nn4}), then we have $m=p^t m_0 \leq n+1$. If $m_0 = 1$, we have $|\alpha^{p^t}| \geq 1$, and then $|\alpha| \geq 1$, it follows by $|\alpha| = |A| \cdot ||u_1\theta_1 + \cdots + u_n\theta_n||$ that
\begin{equation*}
  ||u_1\theta_1 + \cdots + u_n\theta_n|| \geq |A|^{-1} \geq |A|^{-1} \left( \max_{1\leq j \leq n} |u_j| \right)^{-n},
\end{equation*}
and the theorem follows by taking $|\delta| = |A|^{-1}$. If $m_0 > 1$, letting $\alpha'$ denote any one of the $m_0$ (non-identity) algebraic conjugates of $\alpha$, we have
\begin{equation}
  \label{eq:3-22}
  \alpha' = Aw + Au_1\theta_1' + Au_2\theta_2' + \cdots + Au_n \theta_n',
\end{equation}
where $\theta_j'$ is the algebraic conjugate of $\theta_j$. It is easy to see that
\begin{equation*}
  |\alpha'| \leq \max\{|\alpha|, |\alpha - \alpha'|\}
\end{equation*}
and
\begin{equation}
  \label{eq:3-23}
  \begin{split}
      |\alpha' - \alpha| &= |A\sum_{i=1}^n u_i(\theta_i - \theta_i')|\\
& \leq |A| \max\limits_{1\leq i \leq n} |u_i|~|\theta_i - \theta_i'|\\
& \leq |\delta_1| \max\limits_{1\leq i \leq n}|u_i|,
  \end{split}
\end{equation}
where $|\delta_1|: = |A|\max|\theta_i - \theta_i'|$ is a non-zero value independent of $u$. By \eqref{eq:3-20} and \eqref{eq:3-21}, we have $|\alpha| \leq \frac{1}{q}|A|$, and then
\begin{equation}
  \label{eq:3-24}
  |\alpha'| \leq |\delta_2|\max\limits_{1\leq i \leq n} |u_i|,
\end{equation}
where $|\delta_2|$ is independent of $u$ and $|\delta_2| \geq 1$. Since the multiplication of $\alpha$ and its conjugates is a non-zero polynomial, then
\begin{equation*}
|\alpha|^{p^t} |\delta_2|^{p^t(m_0-1)} \left( \max_{1\leq j \leq n} |u_j|  \right)^{p^t(m_0-1)} \geq 1,
\end{equation*}
it follows that
\begin{equation*}
  |\alpha| ~|\delta_2|^{m_0-1} \left( \max_{1\leq j \leq n} |u_j|  \right)^{m_0-1} \geq 1.
\end{equation*}
We note that $|\alpha| = |A|~||u_1\theta_1 + u_2\theta_2 + \cdots + u_n\theta_n||$, then
\begin{equation}
  \label{eq:3-25}
  \begin{split}
    ||u_1\theta_1  + \cdots + u_n\theta_n|| &\geq |A|^{-1} |\delta_2|^{-(m_0-1)}  \left( \max_{1\leq i \leq n} |u_i|  \right)^{-(m_0-1)} \\
&\geq |A|^{-1}|\delta_2|^{-n}(\max\limits_{1\leq i \leq n}|u_i|)^{-n}.
  \end{split}
\end{equation}
Let $|\delta| = |A|^{-1}|\delta_2|^{-n}$. Then we have \eqref{eq:3-19} for all integral points $u\neq 0$, and Theorem~\ref{thm:1-2} follows immediately.
\end{proof}

\section{Binary Quadratic Forms}
\label{sec:binary}

The case of binary quadratic forms over $k_{\infty}$ is similar to the classical case (see Chapter 2 of \cite{ref:6}). However, there still exist some key differences. For example, the compactness lemma (see page 21 of \cite{ref:6}) plays an essential role in  discussion of the classical case, but this result is immediate in the positive characteristic case, since the finiteness of $\fnum_q$ is the natural compactness. To exhibit these differences, we give a detailed discussion here. Throughout this section we shall assume $p>2$, and put
\begin{equation}
  \label{eq:4-1}
  f(x, y) = \alpha x^2 + \beta xy + \gamma y^2, \quad \alpha, \beta, \gamma \in k_{\infty},
\end{equation}
a binary quadratic form of discriminant
\begin{equation*}
  \delta = \beta^2 - 4\alpha \gamma \neq \mbox{perfect square in }k.
\end{equation*}
Let $\theta$ and $\phi$ be the roots of $f(x, 1)=0$, so we have the following factorization into linear factors,
\begin{equation}
  \label{eq:4-2}
  f(x, y) = \alpha L(x, y)M(x, y),
\end{equation}
where
\begin{equation}
  \label{eq:4-3}
  L(x, y) = x-\theta y, \mbox{ and } M(x, y)= x-\phi y
\end{equation}
and
\begin{equation}
  \label{eq:4-4}
  |\alpha (\theta - \phi)| = |\delta|^{\frac{1}{2}}.
\end{equation}
We note that if $\theta\in k_{\infty}$, then $\delta^{\frac{1}{2}}\in k_{\infty}$ and $\phi\in k_{\infty}$.

\begin{lemma}
  \label{lem:4-1} If $a, b \in K$, $(a, b)=1$, and $f(a, b) = \alpha' \neq 0$, then there are two polynomials $c$ and $d$, such that $ad-bc=1$, and
  \begin{equation*}
    f(ax+cy, bx+dy) = \alpha' x^2 + \beta' xy + \gamma' y^2
  \end{equation*}
with $|\beta'| \leq \frac{1}{q}|\alpha'|$.
\end{lemma}

\begin{proof}
  Since $a, b \in K$ and $(a, b) =1$, clearly there are polynomials $c$ and $d$ such that $ad - bc =1$. Let
  \begin{equation*}
    f(ax+cy, bx+dy) = \alpha' x^2 + \beta'' xy + \gamma'' y^2.
  \end{equation*}
It is easy to see that there is an $A\in K$ such that
\begin{equation}
  \label{eq:4-5}
  |\beta'' + A\alpha'|\leq \frac{1}{q}|\alpha'|.
\end{equation}
In fact, one may choose $A = -[\frac{\beta''}{\alpha'}]$, the integral part of $\frac{\beta''}{\alpha'}$. Let $c_1 = c + 2^{-1}Aa$, and $d_1 = d + 2^{-1}Ab$. Then
\begin{equation*}
  ad_1 - bc_1 = ad-bc = 1.
\end{equation*}
Suppose that
\begin{equation*}
      f(ax+c_1y, bx+d_1y) = \alpha' x^2 + \beta' xy + \gamma' y^2,
\end{equation*}
where by \eqref{eq:4-5}
\begin{equation*}
  \beta' = \beta'' + A \alpha', \mbox{ and then }  |\beta'|\leq \frac{1}{q}|\alpha'|.
\end{equation*}
Hence, $c_1$ and $d_1$ are as required, and we complete the proof of the lemma.
\end{proof}

The next lemma is an analogue of the compactness lemma.

\begin{lemma}
  \label{lem:4-2} Let
  \begin{equation*}
      f_n(x, y) = \alpha_n x^2 + \beta_n xy + \gamma_n y^2, \quad (1\leq n \leq \infty),
  \end{equation*}
where $\alpha_n \in K$, $\beta_n \in K$, and $\gamma_n \in K$. Suppose that $\beta_n^2 - 4\alpha_n \gamma_n = \delta$ for all $n\geq 1$ and
\begin{equation}
  \label{eq:4-6}
  0 < |D_1| \leq |\alpha_n| \leq |D_2|, \quad |\beta_n| \leq |D_3|~|\alpha_n|
\end{equation}
for all sufficiently large $n$, where $D_1, D_2, D_3\in K$ are independent of $n$. Then there exist values $\alpha,\beta, \gamma \in K$ with $\beta^2 - 4\alpha\gamma = \delta$, and a subsequence $f_{j_m}(x, y)$ of the $f_n(x,y)$, such that
\begin{equation}
  \label{eq:4-7}
  f_{j_m}(x, y) = \alpha x^2 + \beta xy + \gamma y^2
\end{equation}
for all $j_m$.
\end{lemma}

\begin{proof}
  The hypotheses imply, for large enough $n$, that
  \begin{equation}
    \label{eq:4-8}
    |\beta_n|\leq |D_4|, \mbox{ and } |\gamma_n|\leq |D_5|,
  \end{equation}
for some $D_4, D_5$ in $K$,  independent of $n$. Since $\alpha_n \in K$, $\beta_n \in K$, $\gamma_n \in K$, and the degrees of $\alpha_n$, $\beta_n$, $\gamma_n$ are bounded, the sequence $f_n(x,y)$ runs through a finite set of distinct quadratic forms, and so the assertion of the lemma follows immediately.
\end{proof}

\begin{lemma}
  \label{lem:4-3} Let $f(x, y) = \alpha x^2 + \beta xy + \gamma y^2 $, where $\alpha, \beta, \gamma\in K$, $\delta = \beta^2 - 4\alpha \gamma$ be its discriminant, not a perfect square, and $\theta$, $\phi$ be the roots of $f(x, 1)=0$. Then there is a value $\eta \in k_{\infty}$ with $0 < |\eta| < 1$, and polynomials $a, b, c, d$ with $ad-bc=1$, such that
  \begin{equation}
    \label{eq:4-9}
    L(ax+by, cx+dy) = \eta L(x, y),
  \end{equation}
and
\begin{equation}
  \label{eq:4-10}
  M(ax+by, cx+dy) = \eta^{-1}M(x, y).
\end{equation}
\end{lemma}

\begin{proof}
  Let ${\mathsf S} = \left( \begin{tabular}[c]{cc}
  $a$ & $c$ \\
  $b$ & $d$
  \end{tabular}
 \right)\in SL(2, K)$ be a $2\times 2$ matrix in $K$ with $\det({\mathsf S}) = 1$. We write $X = (x, y)$, $X{\mathsf S} = (ax+by, cx+dy)$.

For any positive integer $n$, we consider the following linear inequalities
\begin{equation}
  \label{eq:4-11}
  \left\{
    \begin{tabular}[l]{l}
      $v(x - \theta y )\geq n+1$, \\
      $v(x-\phi y ) \geq v(\theta - \phi) - n$.
    \end{tabular}
\right.
\end{equation}
By Lemma~\ref{lem:2-4}, there is an integral point $(x, y) = (a_n, b_n)\neq 0$, such that \eqref{eq:4-11} holds. If $a_n = 0$, then $b_n \neq 0$. By \eqref{eq:4-11}, we have $v(\theta b_n) \geq n+1$, and so $v(\theta b_n) > v(\theta)$ when $n$ is large enough. But $v(\theta b_n) = v(\theta) + v(b_n) \leq v(\theta)$ because of $b_n \in K$ and $b_n \neq 0$, which is a contradiction. Hence we have $a_n \neq 0$. Let $A_n = (a_n, b_n)$. Without loss of generality, we may suppose that the greatest common divisor of $a_n$ and $b_n$ is $1$, and $a_n \neq 0$ when $n$ is large enough. By \eqref{eq:4-11}, we have
\begin{equation}
  \label{eq:4-12}
  |L(A_n)|\leq q^{-n-1}, \quad |M(A_n)|\leq q^n|\theta - \phi|,
\end{equation}
and
\begin{equation}
  \label{eq:4-13}
  |f(A_n)| = |\alpha L(A_n)M(A_n)| \leq \frac{1}{q}|\alpha (\theta - \phi)|.
\end{equation}
By Lemma~\ref{lem:4-1}, there is a matrix ${\mathsf S}_n = \left( \begin{tabular}[c]{cc}
  $a_n$ & $c_n$ \\
  $b_n$ & $d_n$
  \end{tabular}
 \right)\in SL(2, K)$ such that
 \begin{equation*}
   f(X{\mathsf S_n}) = \alpha_n x^2 + \beta_n xy + \gamma_n y^2,
 \end{equation*}
where
\begin{equation}
  \label{eq:4-14}
  \alpha_n  = f(A_n), |\beta_n|\leq \frac{1}{q}|\alpha_n|, \mbox{ and } \beta_n^2 - 4\alpha_n \gamma_n = \delta.
\end{equation}
It is easy to see by \eqref{eq:4-13} that $1\leq |\alpha_n| \leq \frac{1}{q}|\alpha(\theta - \phi)|$. By Lemma~\ref{lem:4-2}, there is a subsequence ${\mathsf S}_m$ of ${\mathsf S}_n$, such that
\begin{equation}
  \label{eq:4-15}
  f(X{\mathsf S}_m) = g(x, y) = g(X),
\end{equation}
where $g(x, y)$ is independent of $m$. Let $g(X) = \lambda(X)u(X)$ be any factorization into linear forms. On the other hand,
\begin{equation*}
  g(X) = f(X{\mathsf S}_m) = \alpha L(X{\mathsf S}_m) M(X{\mathsf S}_m).
\end{equation*}
So, we have (by taking a subsequence again if necessary)
\begin{equation}
  \label{eq:4-16}
  L(X{\mathsf S}_m) = \rho_m \lambda(X), \mbox{ and } M(XS_m) = \pi_m U(X),
\end{equation}
where $\rho_m$ and $\pi_m$ are in $k_{\infty}$. But $A_m = (1, 0){\mathsf S}_m$ by construction. Hence on putting $X=(1, 0)$ in \eqref{eq:4-16}, and using \eqref{eq:4-12} we have
\begin{equation*}
  \rho_m / \rho_1 = L(A_m) /  L(A_1) \rightarrow 0, \quad  \text{as $m \to \infty$}.
\end{equation*}
Put $\eta = \rho_m / \rho_1$, so that $0< |\eta| < 1$ when $m$ is large enough. Let ${\mathsf T} = {\mathsf S}_1^{-1}{\mathsf S}_m$. Then
\begin{equation}
  \label{eq:4-17}
  f(X{\mathsf T}) = g(X{\mathsf S}_1^{-1}) = f(X).
\end{equation}
Similarly, \eqref{eq:4-16} gives
\begin{equation*}
  L(X{\mathsf T}) = \rho_m \lambda(X{\mathsf S}_1^{-1})  =\eta L(X)
\end{equation*}
and
\begin{equation*}
  M(X{\mathsf T}) = \eta^{-1} M(X).
\end{equation*}
This proves the lemma with ${\mathsf T} = \left( \begin{tabular}[c]{cc}
  $a$ & $c$ \\
  $b$ & $d$
  \end{tabular}
 \right)$.
\end{proof}

\begin{corollary}
  \label{cor:4-1} Let $x_0, y_0\in K$, and $n$ be an integer. Then there are polynomials $x_1, y_1 \in K$, such that $f(x_1, y_1) = f(x_0, y_0)$, and
  \begin{equation*}
    L(x_1, y_1) = \eta^n L(x_0, y_0), \quad M(x_1, y_1) = \eta^{-n}M(x_0, y_0),
  \end{equation*}
where $0 < |\eta| < 1$.
\end{corollary}

\begin{proof}
Suppose that ${\mathsf T} = {\mathsf S}_1^{-1}{\mathsf S}_m$ in the above proof. We have
  \begin{equation*}
    L(X{\mathsf T}^n) = \eta L(X{\mathsf T}^{n-1}) = \cdots = \eta^n L(X)
  \end{equation*}
and
\begin{equation*}
  M(X{\mathsf T}^n) = \eta^{-1}M(X{\mathsf T}^{n-1}) = \cdots = \eta^{-n} M(X).
\end{equation*}
On writing $X{\mathsf T}^{-1}$ for $X$ in $L(X{\mathsf T}) = \eta L(X)$, we have
\begin{equation*}
  L(X{\mathsf T}^{-1}) = \eta^{-1}L(X),
\end{equation*}
and so, similarly
\begin{equation*}
  L(X{\mathsf T}^n) = \eta^n L(X), \mbox{ and } M(X{\mathsf T}^n) = \eta^{-n} M(X), \mbox{ for } n<0.
\end{equation*}
The corollary follows on putting $(x_1, y_1) = (x_0, y_0){\mathsf T}^n$.
\end{proof}

Next, we give a proof of Theorem~\ref{thm:1-3}. In fact, we will prove a more general result. Recall the definition of the approximation constant $\tau(\theta)$,
\begin{equation}
  \label{eq:4-18}
  \tau(\theta) = \lim_{|Q|\rightarrow \infty} \inf\{|x|~||x\theta||: x\in K \mbox{ and } |x|\geq |Q| \}.
\end{equation}
By Corollary~\ref{cor:2-1}, we have $0\leq \tau(\theta)\leq \frac{1}{q}$. If $\theta$ is a quadratic algebraic element in $k_{\infty}$, then we have the following more accurate result for $\tau(\theta)$.

\begin{theorem}
  \label{thm:4-1} Suppose that $f(x, y) = \alpha x^2 + \beta xy + \gamma y^2$ is a binary quadratic form with $\alpha, \beta, \gamma \in k_{\infty}$ and $\delta = \beta^2 - 4\alpha \gamma$, not a perfect square in $k$. Let $\theta$ and $\phi$ be the roots of $f(x, 1)=0$, and let
  \begin{equation}
    \label{eq:4-19}
    \sigma = \sigma(f) = \inf\{|f(x, y)|: (x, y)\in K^2, (x, y)\neq 0  \}.
  \end{equation}
Then we have
\begin{enumerate}
\item[$(i)$] $\tau(\theta)\geq |\delta|^{-\frac{1}{2}}\sigma$, for any $\alpha$, $\beta$, $\gamma$.
\item[$(ii)$] If $\alpha, \beta, \gamma \in k$ are rational functions, then $\sigma(f)$ is attained by some $(x, y)\in K^2$, $(x, y)\neq 0$, and $\tau(\theta) = |\delta|^{-\frac{1}{2}}\sigma$.
\item[$(iii)$] If $\alpha, \beta, \gamma \in k$, then there are infinitely many $Q\in K$, such that
  \begin{equation}
    \label{eq:4-20}
    |Q|~||Q\theta|| = \tau(\theta).
  \end{equation}
\end{enumerate}
\end{theorem}

\begin{proof}
  We have the following identity
  \begin{equation}\label{eq:4-21}
    \begin{split}
    f(P, Q) &= \alpha (P-\theta Q)(P-\phi Q)  \\
            &= \alpha (\theta - \phi)Q(P-\theta Q) + \alpha(P-\theta Q)^2,
    \end{split}
  \end{equation}
where $(P, Q)\in K^2$. Suppose that $\tau'$ is any number $>\tau(\theta)$. Then there are infinitely many $(P, Q)\in K^2$, such that $|Q|~|Q\theta - P|<\tau'$ with arbitrarily large $|Q|$. Thus, we may assume that $|Q|^{-2}<|\theta - \phi|$, and so
\begin{equation*}
  \begin{split}
    |f(P, Q)| &\leq \max\{|\alpha(\theta-\phi)|\tau', |\alpha|\tau' |Q|^{-2}  \}\\
              &\leq |\alpha(\theta-\phi)|\tau' = |\delta|^{\frac{1}{2}}\tau'.
  \end{split}
\end{equation*}
By the definition of $\sigma$, it follows that
\begin{equation*}
  \sigma \leq |f(P, Q)| \leq |\delta|^{\frac{1}{2}}\tau', \mbox{ and } |\delta|^{-\frac{1}{2}}\sigma \leq \tau'.
\end{equation*}
Therefore, we have $\tau(\theta)\geq |\delta|^{-\frac{1}{2}}\sigma$, which is the assertion of $(i)$.

To prove $(ii)$, if $\alpha, \beta, \gamma \in k$ are rational functions, then there is a polynomial $H$ such that $Hf(x, y)\in K$ for all $(x, y)\in K$. If $(x, y)\in K^2$ and $(x, y)\neq 0$, then $f(x, y)\neq 0$, and then we have $v(Hf(x, y))\leq 0$. Since $v(Hf(x, y))$ takes on integer values,  there is a point $(x_0, y_0)\in K^2$, such that $v(Hf(x_0, y_0)) = \max\{ v(Hf(x, y))\}$. Therefore, $\sigma = |f(x_0, y_0)|$. This proves that $\sigma$ is attained by $|f(x_0, y_0)|$.

To prove  $\tau(\theta) = |\delta|^{-\frac{1}{2}}\sigma$, by Corollary~\ref{cor:4-1} there is an integral point $(P, Q)\in K^2$, such that
\begin{equation*}
  \sigma = |f(x_0, y_0)| = |f(P, Q)|,
\end{equation*}
and $|Q\theta - P|$ is arbitrarily small. It follows by \eqref{eq:4-21} that
\begin{equation}
  \label{eq:4-22}
  \sigma = |f(P, Q)| = |\alpha(\theta - \beta)|~|Q|~|P-Q\theta| \geq  |\delta|^{-\frac{1}{2}}\tau(\theta).
\end{equation}
Hence we have $|\delta|^{-\frac{1}{2}}\sigma \leq \tau(\theta) \leq |\delta|^{-\frac{1}{2}}\sigma$, and so $\tau(\theta) = |\delta|^{-\frac{1}{2}}\sigma$.

\eqref{eq:4-22} also implies the assertion of $(iii)$, since there are infinitely many $(P, Q)\in K^2$, such that
\begin{equation*}
  \sigma = |\alpha(\theta - \phi)|~|Q|~|P-Q\theta|.
\end{equation*}
If follows that
\begin{equation*}
  |Q|~||Q\theta|| \leq |Q|~|P-Q\theta| = |\delta|^{-\frac{1}{2}}\sigma = \tau(\theta) \leq |Q|~||Q\theta||.
\end{equation*}
This completes the proof of Theorem~\ref{thm:4-1}.
\end{proof}

It is easy to see that Theorem~\ref{thm:1-3} follows from parts $(ii)$ and $(iii)$ of the above theorem.

\section{Examples}
\label{sec:exp}

For any $\alpha \in k_{\infty}$, we know that the approximation constant $\tau(\alpha)$ lies in the interval of $0\leq \tau(\alpha)\leq \frac{1}{q}$. In this section, we give two examples involving quadratic algebraic elements. The first example is about a quadratic element whose approximation constant is $\tau(\alpha) = \frac{1}{q^d}$. Of course, this example is well known by using continued fractions as we demonstrate below, but in fact it is a direct consequence of Theorem~\ref{thm:4-1} above. The second example is about the periodic property of continued fractions. Throughout this section, we let $p>2$.

 Example 1. Let $d(T)$ be a polynomial of degree $d$ and $\alpha_1$, $\alpha_2$ be the roots of the  quadratic equation
\begin{equation}
  \label{eq:5-1}
  x^2 + d(T)x-1=0.
\end{equation}
Let $f(x, y)=x^2+d(T)xy - y^2$ be the associated binary quadratic form. Since its discriminant $\delta = d^2(T)+4$, is not a perfect square, and $|\delta| = q^{2d}$, it is easy to verify that the minimal value of $|f(x, y)|$ is attained by $|f(1, 0)|=1$. By Theorem~\ref{thm:4-1}, we have $\tau(\alpha_1) = \tau(\alpha_2) = |\delta|^{-\frac{1}{2}}\sigma = q^{-d}$. In fact, this conclusion also follows by making use of continued fractions as we now demonstrate. We have
\begin{equation*}
  \alpha_i = \frac{1}{2}(-d(T)\pm \sqrt{d^2(T)+4}), \quad (i=1, 2).
\end{equation*}
We first show $\alpha_1$ and $\alpha_2$ have the following continued fraction expansion that
\begin{equation}
  \label{eq:5-2}
  \alpha_1 = \frac{1}{2}(\sqrt{d^2(T)+4}-d(T)) = [0, d(T), d(T), \ldots, d(T), \ldots]
\end{equation}
and
\begin{equation}
  \label{eq:5-3}
  \alpha_2 = \frac{1}{2}(-\sqrt{d^2(T)+4}-d(T)) = -[d(T), d(T), \ldots, d(T), \ldots].
\end{equation}

To prove \eqref{eq:5-2} and \eqref{eq:5-3}, let
\begin{equation}
  \label{eq:5-4}
  \theta = [0, d(T), \ldots, d(T), \ldots].
\end{equation}
We have $[\theta] = 0$, and $\frac{1}{\theta} = d(T)+\theta$, so $\theta$ is a root of \eqref{eq:5-1}. We only show that $\theta = \alpha_1$. Suppose that $x\in k_{\infty}$ with $v(x)\geq 1$. Then $\sqrt{1+x}$ has the following power series expansion
\begin{equation}
  \label{eq:5-5}
  \sqrt{1+x} = 1+\frac{1}{2}x - \frac{1}{8}x^2 + \cdots.
\end{equation}
The series on the right-hand side of \eqref{eq:5-5} is convergent since $v(x)\geq 1$ (see Proposition~2.2 of \cite{ref:8}). We replace $x$ by $4d^{-2}(T)$, to get
\begin{equation*}
  \sqrt{1+\frac{4}{d^2(T)}} = 1 + \frac{2}{d^2(T)} - \frac{2}{d^4(T)} + \cdots .
\end{equation*}
It follows that
\begin{equation*}
  \sqrt{d^2(T)+4} = d(T)+\frac{2}{d(T)} - \frac{2}{d^3(T)}+\cdots .
\end{equation*}
Hence we have $[\sqrt{d^2(T)+4}] = d(T)$. If $\theta\neq \alpha_1$, then $\theta=\alpha_2$, and so we have $2\theta + d(T) = -\sqrt{d^2(T)+4}$, and
\begin{equation}
  \label{eq:5-6}
  [2\theta + d(T)] = [-\sqrt{d^2(T)+4}]  = -d(T).
\end{equation}
On the other hand, $[2\theta + d(T)] = 2[\theta]+ d(T) = d(T)$, which is a contradiction. Therefore we have $\theta = \frac{1}{2}(\sqrt{d^2(T)+4} - d(T))$, and \eqref{eq:5-3} follows because of $\alpha_2 = -\frac{1}{\alpha_1}$.

Suppose now that $\alpha = [a_0, a_1, a_2, \ldots, a_n, \ldots]$ is the continued fraction expansion of an element $\alpha$, where $a_0 = [\alpha]$, and $a_j \in K$, with $\deg a_j\geq 1~(j \geq 1)$. We let
\begin{equation}
  \label{eq:5-7}
  \frac{P_n}{Q_n} = [a_0, a_1, \ldots, a_n], \mbox{ and } A_n = [a_n, a_{n+1}, \ldots],
\end{equation}
where $\frac{P_n}{Q_n}$ is called a convergent of $\alpha$, $a_j$ is called a partial quotient of $\alpha$, and $A_n$ is called a complete quotient of $\alpha$. It is known that (see \cite{ref:9})
\begin{equation}
  \label{eq:5-8}
  \left|\alpha - \frac{P_n}{Q_n}\right|  = \frac{1}{|a_{n+1}|~|Q_n|^2}, \mbox{ and } \left|\alpha - \frac{P}{Q}\right|>\left|\alpha - \frac{P_n}{Q_n}\right|
\end{equation}
for all $P, Q\in K$ with $1\leq |Q| < |Q_n|$. It follows that
\begin{equation}
  \label{eq:5-9}
  \tau(\alpha) = \lim_{n\rightarrow \infty} \inf\left\{ \frac{1}{|a_i|}: i\geq n\right\}.
\end{equation}
By \eqref{eq:5-2} and \eqref{eq:5-3}, we have $\tau(\alpha_1) = \tau(\alpha_2) = q^{-d}$ at once.

We consider the special case where $d(T) = aT+b$, with $a\in \fnum_q^*$ and $b\in \fnum_q$. We denote $\alpha(a, b) = \frac{1}{2}(\sqrt{(aT+b)^2+4} - (aT+b))$. By \eqref{eq:5-2}, we have
\begin{equation}
  \label{eq:5-10}
  \alpha(a, b) = [0, aT+b, \ldots, aT+b, \ldots],
\end{equation}
and $\tau(\alpha(a, b)) = \frac{1}{q}$, so $\alpha(a, b)$ behaves like an analogue of the real number $\frac{1}{2}(\sqrt{5}-1)$. But in the real number case, a real number $\alpha$ has approximation constant $\tau(\alpha) = \frac{1}{\sqrt{5}}$, if and only if $\alpha$ is equivalent to $\frac{1}{2}(\sqrt{5}-1)$. On the function field side, clearly, we may find infinitely many $\alpha\in k_{\infty}$, such that $\tau(\alpha) = \frac{1}{q}$, with $\alpha$ not equivalent to $\alpha(a, b)$. For more about the equivalence in positive characteristic, the reader is referred to \cite{ref:5,ref:13}. This observation suggests that the classical Markov chain does not exist in function fields. For more precise explanation, we refer the readers to \cite[Remark~9.2.1]{ref:n3}.

 Example 2. It is an important question whether the partial quotients in the continued fraction expansion for $\alpha$ are bounded. In such a case, $\alpha$ is said to be badly approximable. In \cite{ref:4,ref:5}, Baum and Sweet showed a famous example in $\fnum_2((\frac{1}{T}))$, with $\alpha$ being the unique solution of the equation $Tx^3+x+T=0$. The partial quotients of the continued fraction expansion of $\alpha$ are all of degree bounded by $2$. In this example, we take a look at quadratic algebraic elements in $k_{\infty}~(p>2)$.  It is known (see \cite{ref:9}, Theorem~3.1) that the sequence of partial quotients in the continued fraction expansion of an $\alpha\in k_{\infty}$ is ultimately periodic, if and only if $\alpha$ is a quadratic algebraic element over $k$. For such $\alpha$ we write
\begin{equation}
  \label{eq:5-11}
  \alpha = [a_0, a_1, \ldots, a_n, \overline{b_1, b_2, \ldots, b_m}],
\end{equation}
and define $D(\alpha)$ by
\begin{equation}
  \label{eq:5-12}
  D(\alpha) := \max\limits_{1\leq i \leq m}\deg(b_i).
\end{equation}
Since the approximation constant of $\alpha$ is $\tau(\alpha) = q^{-D(\alpha)}$, by Theorem~\ref{thm:4-1}, we have
\begin{equation}
  \label{eq:5-13}
  D(\alpha) = \frac{1}{2}\deg \delta + t(f),
\end{equation}
where $f(x, y) = ax^2 + bxy + cy^2$, $a, b, c\in K$, $\delta = b^2-4ac$, $\alpha$ is a root of $f(x, 1)=0$, and
\begin{equation*}
  t(f) := \max\{v(f(x, y)): (x, y)\in K^2, (x, y)\neq 0\}.
\end{equation*}
Formula \eqref{eq:5-13} shows that the largest degree of the partial quotients only depends on the discriminant $\delta$, and the maximum value of $v(f(x, y))$. In particular, we have $D(\alpha)\leq \frac{1}{2}\deg \delta$, and $D(\alpha) = \frac{1}{2}\deg \delta$ if $\alpha$ is a quadratic algebraic integral element over $k$.

~\\
\noindent Acknowledgements. I would like to thank Professor Jacques Peyriere for a valuable discussion about the Haar integral formula for a linear change of variables. In fact, the proof of Lemma~2.2 given here was suggested in an  email letter he sent to me. I would like to thank Professor Todd Cochrane for his careful reading throughout the manuscript and correcting some of the mistakes in English. I also would like to thank the referees for their valuable comments.

\end{document}